\documentclass{amsart}

\usepackage[english]{babel}

\usepackage{mathrsfs}
\usepackage{amsmath}
\usepackage{amsfonts}
\usepackage{amssymb}
\usepackage{amsthm}
\usepackage{amsrefs}
\usepackage[colorlinks]{hyperref}
\hypersetup{allcolors=blue}
\usepackage{indentfirst}
\usepackage{enumerate}
\usepackage{float}
\usepackage{graphicx}
\usepackage{color}
\usepackage[usenames,dvipsnames,svgnames,table,x11names]{xcolor}
\usepackage{dsfont}
\usepackage{array}
\usepackage{comment}
\usepackage[all,knot,arc,import,poly]{xy}
\usepackage{textcomp}
\usepackage[colorinlistoftodos]{todonotes}
\usepackage{multicol}
\usepackage{tocvsec2}

\newtheorem{thm}{Theorem}[section]
\newtheorem{rmk}[thm]{Remark}

\newtheorem{prop}[thm]{Proposition}
\newtheorem{cor}{Corollary}[thm]
\newtheorem{lema}[thm]{Lemma}
\newtheorem{defi}[thm]{Definition}
\newtheorem{exe}[thm]{Example}

\newtheorem*{teo}{Theorem}

\everymath{\displaystyle}

\emergencystretch=\maxdimen
\def\C{\mathbb{C}}

\def\P{\mathbb{P}}
\def\calP{\mathcal{P}}
\def\calL{\mathcal{L}}
\def\calN{\mathcal{N}}
\def\X{\mathscr{X}_{k,d,n}}

\title[GIT for linear systems of hypersurfaces]{On the GIT stability of linear systems of hypersurfaces in projective space}

\author[M. Hattori]{Masafumi Hattori \textsuperscript{1}}
\address{\textsuperscript{1} Department of Mathematics, Faculty of Science, Kyoto University, Kyoto, 606-8502, Japan}
\email{hattori.masafumi.47z@st.kyoto-u.ac.jp}

\author[A. Zanardini]{Aline Zanardini \textsuperscript{2}}
\address{\textsuperscript{2} Mathematical Institute, Leiden University, Niels Bohrweg 1, Leiden, The Netherlands}
\email{a.zanardini@math.leidenuniv.nl}
\date{\today}

\begin{document}

\maketitle

\begin{abstract}
    We consider the problem of classifying linear systems of hypersurfaces (of a fixed degree) in some projective space up to projective equivalence via geometric invariant theory (GIT). We provide an explicit criterion that solves the problem completely. As an application, we consider a few relevant geometric examples recovering, for instance,  Miranda's description of the GIT stability of pencils of plane cubics. Furthermore, we completely describe the GIT stability of Halphen pencils of any index.
\end{abstract}

\setcounter{tocdepth}{2}
\tableofcontents

\section{Introduction}
Let $\mathscr{X}_{k,d,n}$ denote the space of $k$-dimensional linear systems of hypersurfaces of degree $d$ in $\P^n$. In this paper we investigate and solve the problem of parametrizing points in $\X$ up to projective equivalence via geometric invariant theory (GIT). 

In particular, we provide alternative descriptions of the results in \cite{stab}, \cite{quartics}, \cite{azstab}, \cite{quadricsMir}, \cite{wall}, \cite{cubicSurfaces}, and \cite{genus5}. Furthermore, since every pencil of plane curves can be seen as a holomorphic foliation of $\mathbb{P}^2$, we also provide an alternative description of the results in \cite{ca1}, \cite{ca2} and \cite{ca3}. 

If we let $V$ denote the space of sections $H^0(\mathbb{P}^n,\mathcal{O}_{\mathbb{P}^n}(1))$, then the space $\mathscr{X}_{k,d,n}$ can be identified with the Grassmannian of $k$-planes in $S^{d}V^{\ast}$, which can be embedded in $\mathbb{P}(\Lambda^{k+1} S^{d}V^{\ast})\simeq \mathbb{P}^N$ via Pl\"{u}cker coordinates. The group $PGL(V)$ acts naturally on $V$, hence on the invariant subvariety $\mathscr{X}_{k,d,n} \subset \mathbb{P}^N$, and the problem we are interested in consists in constructing the corresponding quotient. But since the group $SL(V)$ also acts naturally on $\X$ with the same orbits as $PGL(V)$, the GIT machinery tells us it suffices to consider the action of the former and to describe what are the (semi)stable points of $\mathscr{X}_{k,d,n}\subset \P^N$ with respect to such action. 

Observing that $SL(V)$ also acts on $S^{d}V^{\ast}, S^{d(k+1)}V^{\ast}$ and on $\mathscr{X}_{k-1,d,n}$, we prove:

\begin{thm}[= Theorem \ref{mainTHM}]
A linear system $\calL \in \X$ is  GIT stable (resp. semistable) if and only if for any choice of generators $H_1,\ldots,H_{k+1}\in \calL$ the hypersurface $H_1+\ldots+H_{k+1}$ of degree $d(k+1)$ is GIT stable (resp. semistable).
\label{thmINTRO}
\end{thm}

In other words, we show the two questions
\begin{quote}
    \center \it When is a linear  system $\calL \in \X$ GIT (semi)stable? 
\end{quote}
and
\begin{quote}
    \center \it  When is the union of $k+1$ (distinct) hypersurfaces of degree $d$ in $\P^n$ GIT (semi)stable? 
\end{quote}
are equivalent. 

In particular, when $d(k+1)\geq n+1$, we obtain:

\begin{cor}[= Corollary \ref{mainCOR}]
If $\calL\in \X$ is GIT non-stable (resp. unstable), then there exist $k+1$ (distinct) hypersurfaces $H_1,\ldots,H_{k+1}\in \calL$ such that
\[
lct(\P^n,H_1+\ldots+H_{k+1})\leq \frac{n+1}{d(k+1)} \quad (\text{resp.}\,<)
\]
where $lct(\P^n,H_1+\ldots+H_{k+1})$ denotes the log canonical threshold of the pair $(\P^n,H_1+\ldots+H_{k+1})$ (see Section \ref{sec:lct} for details).
\end{cor}

And, if we consider only linear systems $\calL \in \X$ that have base locus of dimension $n-k-1$, which we call \textit{regular}, then we further show:

\begin{cor}[= Corollary \ref{chow}]
Let $k\leq n-1$ and let $\calL\in \X$ be a regular linear system. If $\calL$ is generated by $H_{f_1},\ldots, H_{f_{k+1}}$ and the complete intersection $H_{f_1} \cap\ldots \cap H_{f_{k+1}} \subset \P^n$ is Chow (semi)stable, then $\calL$ is GIT (semi)stable.
\end{cor}

Moreover, we also prove:

\begin{thm}[= Corollary \ref{mainCOR2}]
Let $k>1$ and let $\calL\in \X$ be a linear system which contains at least one GIT semistable hypersurface. If $\calL$ is GIT non-stable (resp. unstable) then there exists a sub-linear system of dimension $k-1$ which is GIT non-stable (resp. unstable).  
\end{thm}

As an application of our results we revisit the work of Miranda on pencils of plane cubics \cite{stab} (Section \ref{cubics}) and the work of Wall on nets of conics \cite{wall} (Section \ref{conics}), and we explicitly explain how one can recover their stability criteria  using Theorem \ref{thmINTRO}. Moreover, in Section \ref{halphen} we further describe the stability of so-called Halphen pencils of index $m$ completely. Concretely, we prove the following:

 \begin{thm}[= Theorem \ref{mainHalphen} + Theorem \ref{mainHalphen2}]
Let $\calP$ be a Halphen pencil of index $m$ and denote by $Y$ the corresponding rational elliptic surface. If $lct(Y,F)>\frac{1}{2m}$ (resp. $\ge$) for any fiber $F$, then $\calP$ is GIT stable (resp. semistable).

Furthermore, the converse also holds except for a few GIT stable Halphen pencils of index $m=2$, which are given by \cite[Examples 7.46, 7.47, 7.55]{azconstr}; and of index $m=3$, which are given by Example \ref{exc.case3}.
\label{MainThmHal}
\end{thm}

In particular, it follows from \cite[Corollary H]{Hat}, that the notion of GIT stability for Halphen pencils is closely related to the notions of adiabatic K-stability and log-twisted K-stability of the base curve of the corresponding rational elliptic surfaces as introduced in \cite{Hat}. Indeed, log twisted K-stability of the base curve implies GIT stability for the Halphen pencil.

 Finally, we remark that the GIT stability of regular linear systems has also been recently considered in \cite{compint}. After finishing this manuscript, we were made aware that Papazachariou had also generalized the work in \cite{azdegd} and that, independently, he had obtained some of the same results from Sections \ref{narb} and \ref{sec:crit}.

We work over $\C$ throughout.

\subsection*{Acknowledgments} We thank Professors Antonella Grassi and Yuji Odaka for their valuable comments on an earlier version of the manuscript. And we also thank Theodoros S. Papazhachariou for letting us know about his work. M.H. was partially 
supported by JSPS KAKENHI 	22J20059  
(Grant-in-Aid for JSPS Fellows DC1).

\section{Notations and relevant background}\label{narb}

A fundamental tool in the analysis of GIT stability is the numerical criterion of Hilbert-Mumford (see Proposition \ref{HM}). In our setting, it comprises in understanding how one-parameter subgroups of $SL(V)$ act on the Pl\"{u}cker coordinates and studying the sign of the corresponding weights.

Throughout the paper we will always assume any one-parameter subgroup $\lambda$ of $SL(V)$ is normalized, meaning we choose coordinates $(x_0:x_1:\ldots:x_n)$ in $\mathbb{P}^n$ so that $\lambda: \C^{\times} \to SL(V)$ is given by 
\begin{equation}
t \mapsto \begin{pmatrix}
t^{a_0} &  &  &    \\
& t^{a_1} &  &    \\ 
 &  & \ddots &    \\
 &  &  & t^{a_n} \end{pmatrix}
\label{normal}
\end{equation}
for some $a_i \in \mathbb{Z}$ such that $a_0 \geq a_1 \geq \ldots \geq a_n, a_0>0$ and $a_0+a_1+\ldots+a_n=0$.

Therefore, if we choose $\calL \in \X$ and $k+1$ hypersurfaces $H_{f_1}\ldots H_{f_{k+1}}$ as generators, each $H_{f_j}$ represented (in the fixed coordinates) by 
$f_j =\sum_{I} f^j_Ix_I=0$, then the
action of $\lambda(t)$ on the Pl\"{u}cker coordinate 
$M_{I_1\ldots I_{k+1}}$ is given by 

%\doteq \sum_{j=1}^{k+1}(-1)^{i+j}f^i_{I_j}M_{ij}

\[
M_{I_1\ldots I_{k+1}} \mapsto t^{a_{I_1\ldots I_k}}M_{I_1\ldots I_{k+1}},
\]

where:
\begin{itemize}
\item $I=(i_0,\ldots,i_{n-1})$ is a tuple of non-negative
integers,
\item $x_I$ denotes the monomial $x_0^{i_0}\cdot x_1^{i_1}\cdot \ldots \cdot x_{n-1}^{i_{n-1}} \cdot x_n^{d-i_0-\ldots-i_{n-1}}$,
\item $M_{I_1\ldots I_{k+1}}$ denotes the determinant of the matrix below 
\begin{equation}
 \begin{pmatrix}
 f^1_{I_1} & f^1_{I_2} & \ldots & f^1_{I_{k+1}} \\
 f^2_{I_1} & f^2_{I_2} & \ldots & f^2_{I_{k+1}} \\
 \vdots & \vdots & \ddots & \vdots \\
 f^{k+1}_{I_1} & f^{k+1}_{I_2} & \ldots & f^{k+1}_{I_{k+1}}
 \end{pmatrix}
 \label{matrix}
 \end{equation}
 \item $I_j=({i_0}_j,\ldots, {i_{n-1}}_j)$ for $j=1,\ldots, k+1$,
     \item and we define 
    \[
    a_{I_1\ldots I_{k+1}} \doteq \sum_{l=0}^{n-1} a_l(i_{l_1}+\ldots+i_{l_{k+1}}) +a_n\left(d(k+1)-\sum_{l=0}^{n-1}(i_{l_1}+\ldots+i_{l_{k+1}})\right)
    \]
\end{itemize}

In particular, letting 

\begin{itemize}
    \item $ \tilde{a}_{I_1\ldots I_{k+1}} \doteq \sum_{l=0}^{n-1} (a_l-a_n)(i_{l_1}+\ldots+i_{l_{k+1}})$,
    \item $A_{\lambda} \doteq \sum_{l=0}^{n-1} (a_l-a_n)=-a_n(n+1)$, and
    \item $\omega(\mathcal{L},\lambda)\doteq \min\{\tilde{a}_{I_1\ldots I_{k+1}}\,\,:\,\,M_{I_1\ldots I_{k+1}}\neq 0\}$
\end{itemize}

\vspace{0.5cm}

 the Hilbert-Mumford criterion becomes:

\begin{prop}[Hilbert-Mumford criterion]
A point $\calL\in \X$ is GIT unstable (resp. non-stable)  if and only if there exists a one-parameter subgroup $\lambda$ of $SL(V)$ such that 
\[
\frac{d(k+1)}{n+1} < \frac{\omega(\mathcal{L},\lambda)}{A_{\lambda}}  \quad (\text{resp.}\,\leq)
\]
\label{HM}
\end{prop}

\begin{proof}
A point $\calL\in \X$ is unstable (resp. non-stable)  if and only if there exists a one-parameter subgroup $\lambda: \mathbb{C}^{\times} \to SL(V)$ such that whenever $M_{I_1\ldots I_k}\neq 0$ we have 
\[
a_{I_1\ldots I_{k+1}} >0 \quad (\text{resp.}\,\geq)
\]
if and only if 
\[
\sum_{l=0}^{n-1} a_l(i_{l_1}+\ldots+i_{l_{k+1}}) +a_n\left(d(k+1)-\sum_{l=0}^{n-1}(i_{l_1}+\ldots+i_{l_{k+1}})\right)>0 \quad (\text{resp.}\,\geq)
\]
if and only if
\[
\sum_{l=0}^{n-1}(a_l-a_n)(i_{l_1}+\ldots+i_{l_{k+1}})-\frac{d(k+1)}{n+1}\cdot A_{\lambda}>0 \quad  \quad (\text{resp.}\,\geq)
\]
if and only if
\[
\frac{\tilde{a}_{I_1\ldots I_{k+1}}}{A_{\lambda}} >   \frac{d(k+1)}{n+1} \quad (\text{resp.}\,\geq)
\]

\end{proof}

Similarly, we can define 
\[
\omega(f_j,\lambda)\doteq \min\left\{\sum_{l=0}^{n-1}(a_l-a_n)\cdot i_l\,\,:\,\,f^j_{I}\neq 0\right\}
\]
which, as in \cite{azdegd},  allow us to somewhat easily relate the GIT stability of $\calL \in \X$ to the GIT stability of the hypersurfaces $H_{f_j}$ lying on $\calL$ and, in particular, to the number $lct(\mathbb{P}^n,H_{f_j})$, which is known as the log canonical threshold of the pair $(\mathbb{P}^n,H_{f_j})$ (see e.g. \cite[Section 8]{singpairs}).

With this in mind, for the convenience of the reader, we present next some basic notions concerning log canonical pairs and toric valuations. And we further recall some results on GIT-stability of hypersurfaces, which can also be deduced from Theorem \ref{mainTHM} ($k=0$ case). 

\subsection{The log canonical threshold}
\label{sec:lct}

Let $X$ be a smooth projective variety and let $E$ be a prime divisor over $X$, that is, a prime divisor on a smooth variety $Y$ that admits a proper birational morphism $\pi:Y\to X$. Then we define its log discrepancy to be the number
 \[
 A_{X}(E)=1+\mathrm{ord}_E(K_{Y/X})
 \]
 where $\mathrm{ord}_E(K_{Y/X})$ denotes the coefficient of $E$ in $K_{Y/X}=K_{Y}-\pi^*K_X$.
 
 We then say a pair $(X,D)$ is log canonical, for $D$ an effective $\mathbb{Q}$-divisor on $X$, if for any such $\pi:Y\to X$ one has $A_{X}(E)-\mathrm{ord}_E(D)\ge0$ for any prime divisor $E$ on $Y$, where $\mathrm{ord}_E(D)$ denotes the coefficient of $E$ in $\pi^*D$. Furthermore, we define the log canonical threshold of the pair $(X,D)$ as the quantity $$lct(X,D)=\inf_E\frac{A_{X}(E)}{\mathrm{ord}_E(D)}$$ where $E$ runs over all prime divisors over $X$ (i.e. all prime divisors $E$ on some model $Y$ as above) such that $\mathrm{ord}_E(D)\ne0$. 
 
 It is well-known that the above infimum is actually a minimum and we say $E$ computes $lct(X,D)$ if $lct(X,D)=\frac{A_{X}(E)}{\mathrm{ord}_E(D)}$.
 
In general, given $X$ and $E$ a prime divisor over $X$ as above, we can define a valuation $\mathrm{ord}_E$ on $X$ that sends each rational function in $K(X)^{\times} = K(Y)^{\times}$ to its order of vanishing along $E$. In particular, 

\begin{defi}
    When $X=\P^n$ we will say a prime divisor $E$ over $\P^n$ is \textbf{toric} if $\mathrm{ord}_E$ is a toric valuation, i.e. it is $T$-invariant with respect to some maximal torus $T\subset SL(V)$. 
    \label{DefToricDiv}
\end{defi}

\begin{rmk}\label{fac1}
    If we consider $X=\P^n$ as a toric variety, with respect to some maximal torus $T\subset SL(V)$, and choose a point $p$ in $\P^n$ that lies in the intersection of two (distinct) divisors which are invariant under the torus action, then any exceptional divisor of a blow-up at $p$ is toric in the sense of Definition \ref{DefToricDiv}.
\end{rmk}

Then we observe the following holds (cf. \cite[Section 8]{singpairs} and \cite[Proposition 4.2]{azdegd}):

 \begin{lema}\label{toric}
Let $H_f:(f=0)\subset\mathbb{P}^n$ be a hypersurface of degree $d$ and let $\lambda$ be a one-parameter subgroup of $SL(V)$. Then there exists a toric prime divisor $E$ over $\mathbb{P}^n$ such that (see Remark \ref{stabhyp})
\[
\omega(f,\lambda)=\frac{\mathrm{ord}_E(H_f)}{A_{\mathbb{P}^n}(E)}.
\]
\end{lema}

\begin{proof}
As in (\ref{normal}), we choose coordinates in $\P^n$ that diagonalize $\lambda$. This corresponds to a choice of maximal torus in $SL(V)$ and we can thus consider $\mathbb{P}^n$ as a toric variety. In particular, $\omega(-)\doteq\omega(-,\lambda)$ is a toric valuation (see e.g. \cite{fulton}) and by \cite[Proposition 5.1]{JM} it follows that there exists a toric prime divisor $E$ over $\mathbb{P}^n$ and $c>0$ such that $\omega(-)=c\cdot\mathrm{ord}_E(-)$ and $\omega(f,\lambda)=\frac{\mathrm{ord}_E(H_f)}{A_{\mathbb{P}^n}(E)}$.
\end{proof}

In particular,

\begin{cor}[{\cite{hacking}, \cite{kimlee}}]\label{hypersurftoric}
If $H$ is GIT unstable (resp. non-stable) (see Remark \ref{stabhyp}) then 
\[
\frac{d}{n+1}<\frac{1}{lct(\mathbb{P}^n,H)} \qquad (\text{resp.}\, \leq)
\]

Moreover, the converse also holds if a toric divisor over $\mathbb{P}^n$ computes $lct(\mathbb{P}^n,H)$.
\end{cor}

\begin{rmk}
Note that the inequality $lct(\mathbb{P}^n,H) \leq \underset{E:\textrm{toric}}{\mathrm{inf}}\,\,\,\frac{A_{\mathbb{P}^n}(E)}{\mathrm{ord}_E(H)}$ can be strict in general. 

For example, for any P\l oski curve $C_d$ of even degree $d$ we have $lct(X;C_d)=\frac{5}{2d}$ but $\underset{E:\textrm{toric}}{\mathrm{inf}}\,\,\,\frac{A_{\mathbb{P}^n}(E)}{\mathrm{ord}_E(H)}=\frac{3}{d}$ since $C_d$ is known to be strictly GIT semistable (cf.~\cite{lct}). 
\end{rmk}

\begin{rmk}
Note also that for hypersurfaces of degree $d$ in $\P^n$, the Hilbert-Mumford criterion says a hypersurface $H_f:(f=0)$ is GIT unstable (resp. non-stable)  if and only if there exists a one-parameter subgroup $\lambda:\mathbb{C}^{\times} \to SL(V)$ such that 
\[
\frac{d}{n+1} < \frac{\omega(f,\lambda)}{A_{\lambda}} \quad (resp. \leq)
\] 
where, in general, we write $f=\sum_If_Ix_I=0$ and we define 
\[
\omega(f,\lambda)\doteq \min\left\{\sum_{l=0}^{n-1}(a_l-a_n)i_l\,\,:\,\,f_{I}\neq 0\right\}.
\]
\label{stabhyp}
\end{rmk}

\section{The stability criterion}
\label{sec:crit}

Adapting the work in \cite{azdegd} by the second author, we will now prove Lemmas \ref{keylemma2} and \ref{keylemma} below, which generalize \cite[Proposition 3.5]{azdegd} and  \cite[Corollary 3.11]{azdegd}. These two lemmas are the main ingredients in the proof of Theorem \ref{thmINTRO}.

\begin{lema}
Given $\calL\in \X$ and any $l \leq k+1$ (distinct) hypersurfaces $H_{f_1},\ldots,H_{f_{l}}\in \mathcal{L}$ we have that
\[
\sum_{j=1}^{l}\omega(f_j,\lambda)  \leq \omega(\mathcal{L},\lambda)
\]
for all one-parameter subgroups $\lambda:\mathbb{C}^{\times} \to SL(V)$.
\label{keylemma2}
\end{lema}

\begin{proof}
Fix $\calL\in \X$ and any $l \leq k+1$ (distinct) hypersurfaces $H_{f_1},\ldots,H_{f_{l}}\in \mathcal{L}$. Choose $H_{f_{l+1}},\ldots, H_{f_{k+1}}$ so that we obtain a set of generators for $\calL$.

Choose $\lambda$ and fix $I_1,\ldots ,I_{k+1}$ such that $\omega(\calL,\lambda)=\tilde{a}_{I_1\ldots I_{k+1}}$. Then $M_{I_1\ldots I_{k+1}}\neq 0$ and, by definition, the determinant of $(\ref{matrix})$ is non-zero. In particular, each row $i$ and each column $j$ of $(\ref{matrix})$ must contain a non-zero element. Thus, there exists a permutation $\sigma \in S_{k+1}$ such that $f^{\sigma(j)}_{I_j}\neq 0$ for all $j=1,\ldots k+1$, and it follows that
\begin{align*}
\omega(\calL,\lambda)&=\tilde{a}_{I_1\ldots I_{k+1}} \\
&= \sum_{l=0}^{n-1} (a_l-a_n)(i_{l_1}+\ldots+i_{l_{k+1}})\\
&= \sum_{j=1}^{k+1}\underbrace{\left(\sum_{l=0}^{n-1} (a_l-a_n)\cdot i_{l_j}\right)}_{\geq  \omega(f_{\sigma(j)},\lambda)}\\
&\geq \sum_{j=1}^{k+1} \omega(f_{j},\lambda)  \geq \sum_{j=1}^{l} \omega(f_{j},\lambda).
\end{align*}
\end{proof}

\begin{lema}
Given $\calL\in \X$, a one-parameter subgroup $\lambda$ of $SL(V)$ and any hypersurface $H_{f_1} \in \mathcal{L}$, there exist $H_{f_2},\ldots,H_{f_{k+1}} \in \calL$ ($\neq H_{f_1}$) such that
\[
\omega(\mathcal{L},\lambda) = \sum_{j=1}^{k+1}\omega(f_j,\lambda).
\]
\label{keylemma}
\end{lema}

\begin{proof}
Since the Pl\"{u}cker coordinates are (up to scalar multiplication) independent of the choice of generators, we can argue as follows: Fix  a one-parameter subgroup $\lambda:\mathbb{C}^{\times} \to SL(V)$ and choose generators $H_{f_1},\ldots,H_{f_{k+1}}$ for $\mathcal{L}$.

Now, let $I_1$ be such that $f^1_{I_1}\neq 0$ and $\omega(f_1,\lambda)=\sum_{l=0}^{n-1} (a_l-a_n)\cdot i_{l_1}$.

Replacing $f_j$ by $f_j'=f_j-\frac{f^j_{I_1}}{f^1_{I_1}}f_1$ for all $j=2,\ldots, k+1$ we have $f^j_{I_1}=0$ for all $j=2,\ldots, k+1$. In particular, $M_{I_1\ldots I_{k+1}}=f^1_{I_1}M_{11}\neq 0$ for all $I_2,\ldots,I_{k+1}$ such that $M^{(1)}_{11}\neq 0$, where $M^{(1)}_{11}$ is the determinant of the matrix obtained from $(\ref{matrix})$ by removing the first row and the first column.

Next, let $I_2$ be such that $f^2_{I_2}\neq 0$ and $\omega(f_2,\lambda)=\sum_{l=0}^{n-1} (a_l-a_n)\cdot i_{l_2}$. Again, replacing $f_j$ by $f_j'=f_j-\frac{f^j_{I_2}}{f^2_{I_2}}f_2$ for all $j=3,\ldots, k+1$ we have $f^j_{I_2}=0$ for all $j=3,\ldots, k+1$. In particular, $M^{(1)}_{11}\neq 0$ whenever $M^{(2)}_{11}\neq 0$, where  $M^{(2)}_{11}$ is the sub-determinant of $M^{(1)}_{11}$ obtained by removing the first row and the first column.

Proceeding this way, up to replacing $H_{f_j}, j=2,\ldots,k+1$  ($j-1$ times), we can find $I_1,\ldots, I_{k+1}$ such that $\omega(f_j,\lambda)=\sum_{l=0}^{n-1} (a_l-a_n)\cdot i_{l_j}$ for $j=1,\ldots,k+1$ and $M_{I_1\ldots I_{k+1}}\neq 0$.

In other words, given $H_{f_1}$, we can find $H_{f_2},\ldots, H_{f_{k+1}} \neq H_{f_1}$ all distinct and $I_1,\ldots,I_{k+1}$ such that 
\[
M_{I_1\ldots I_{k+1}} =
\begin{vmatrix}
 f^1_{I_1} & f^1_{I_2} & f^1_{I_3} & \ldots & f^1_{I_{k+1}} \\
 0 & f^2_{I_2} & f^2_{I_3} & \ldots & f^2_{I_{k+1}} \\
 0 & 0 & f^3_{I_3} & \ldots & f^3_{I_{k+1}} \\
  \vdots & \vdots & \vdots & \ddots & \vdots \\
 0 & 0 & 0 & \ldots & f^{k+1}_{I_{k+1}}
 \end{vmatrix} \neq 0
\]
and $\omega(f_j,\lambda)=\sum_{l=0}^{n-1} (a_l-a_n)\cdot i_{l_j}$ for $j=1,\ldots,k+1$

Therefore,
\begin{align*}
\omega(\calL,\lambda) &\leq \tilde{a}_{I_1\ldots I_{k+1}} \\ &= \sum_{l=0}^{n-1} (a_l-a_n)(i_{l_1}+\ldots+i_{l_{k+1}}) \\
&= \sum_{j=1}^{k+1}\left(\sum_{l=0}^{n-1} (a_l-a_n)\cdot i_{l_j}\right) \\
&= \sum_{j=1}^{k+1}\omega(f_j,\lambda) 
\end{align*}
and,  in fact, by Lemma \ref{keylemma2} equality holds.
\end{proof}
As a consequence, using both Lemmas \ref{keylemma2} and \ref{keylemma} we can finally prove:

\begin{thm}
A point $\calL\in \X$ is  GIT stable (resp. semistable) if and only if for any $k+1$ (distinct) hypersurfaces $H_1,\ldots,H_{k+1}\in \calL$ the degree $d(k+1)$ hypersurface $H_1+\ldots+H_{k+1}$  is GIT stable (resp. semistable).
\label{mainTHM}
\end{thm}

In our proof, a crucial observation is that: if $H_f:f=0$ and $H_g:g=0$ are any two hypersurfaces of degree $d$ in $\P^n$, then $\omega(f\cdot g,\lambda)=\omega(f,\lambda)+\omega(g,\lambda)$ for any one-parameter subgroup $\lambda:\C^{\times} \to SL(V)$. In fact, as already mentioned, for any $\lambda$ the function $\omega(-,\lambda)$ defines a (toric) valuation.

The argument is as follows:

\begin{proof}[Proof of Theorem \ref{mainTHM}]
First, assume $\calL$ is GIT stable (resp. semistable) and pick a one-parameter subgroup $\lambda$. Given $H_{f_1},\ldots,H_{f_{k+1}} \in \calL$ (all distinct), by Lemma \ref{keylemma2} and Proposition \ref{HM} we have:
\begin{align*}
\frac{d(k+1)}{n+1}> (\text{resp.}\, \geq) \frac{\omega(\calL,\lambda)}{A_{\lambda}} &\geq \sum_{j=1}^{k+1} \frac{\omega(f_j,\lambda)}{A_{\lambda}}\\
&= \frac{\omega(f_1\cdot \ldots \cdot f_{k+1},\lambda)}{A_{\lambda}}
\end{align*}
and we conclude $H_{f_1}+\ldots +H_{f_{k+1}}$ is GIT stable (resp.~semistable) (see Remark \ref{stabhyp}).

For the converse, choose a one-parameter subgroup $\lambda$ and $H_{f_1}\in \calL$. By Lemma \ref{keylemma} we can find $H_{f_2},\ldots, H_{f_{k+1}}\in \calL$ such that:
\begin{align*}
\frac{\omega(\mathcal{L},\lambda)}{A_{\lambda}} &=  \sum_{j=1}^{k+1}\frac{\omega(f_j,\lambda)}{A_{\lambda}}\\
&= \frac{\omega(f_1\cdot \ldots \cdot f_{k+1},\lambda)}{A_{\lambda}}
\end{align*}
Now, because $H_{f_1}+\ldots+H_{f_{k+1}}$ is GIT stable (resp. semistable), we have (see Remark \ref{stabhyp})
\[
\frac{\omega(f_1\cdot \ldots \cdot f_{k+1},\lambda)}{A_{\lambda}} <\frac{d(k+1)}{n+1} \quad (\text{resp.}\, \leq)
\]
and it follows that $\calL$ is GIT stable (resp. semistable) by Proposition \ref{HM}.
\end{proof}

Some direct consequences of Theorem \ref{mainTHM} and its proof are:

\begin{cor}
If $\calL\in \X$ is GIT non-stable (resp. unstable), then there exist $k+1$ (distinct) hypersurfaces $H_1,\ldots,H_{k+1}\in \calL$ such that
\[
lct(\P^n,H_1+\ldots+H_{k+1})\leq \frac{n+1}{d(k+1)} \quad (\text{resp.}\,<)
\]
\label{mainCOR}
\end{cor}

\begin{proof}
It follows from \cite{hacking} and \cite{kimlee}. See also Lemma \ref{toric}.
\end{proof}

\begin{cor}
Let $\calP\in \mathscr{X}_{1,d,n}$ be a pencil and choose $H_{f}\in \calP$. Then $\calP$ is GIT non-stable (resp. unstable) if and only if there exists $H_g\in \calP$ ($\neq H_f$) such that $H_f+H_g$ is GIT non-stable (resp. unstable). 
\label{keycor}
\end{cor}

\begin{cor}
Let $k>1$ and let $\calL\in \X$ be a linear system which contains at least one GIT semistable hypersurface (e.g. a smooth one). If $\calL$ is GIT non-stable (resp. unstable) then there exists a sub-linear system of dimension $k-1$ which is GIT non-stable (resp. unstable). 
\label{mainCOR2}
\end{cor}

And:

\begin{cor}
Let $k\leq n-1$ and let $\calL\in \X$ be a regular linear system generated by $H_{f_1},\ldots, H_{f_{k+1}}$. If the complete intersection $H_{f_1} \cap\ldots \cap H_{f_{k+1}} \subset \P^n$ is Chow stable, then $\calL$ is GIT stable.
\label{chow}
\end{cor}

\begin{proof}
It follows from \cite[Theorem 1.1]{sano}. But an alternative proof is also provided on Appendix \hyperref[app]{A}. 
\end{proof}

Finally, in view of the criterion given by Theorem \ref{mainTHM} (and its Corollaries), we end this Section by observing we have the following partial stability criteria for hypersurfaces of degree $\alpha$ in $\P^n$:

\begin{prop}
Let $H=m_1H_1+\ldots + m_{j}H_{j}$ be a (possibly) reducible hypersurface of degree $\alpha$ in $\P^n$, where each $H_i$ is irreducible of degree $d_i$. If we can find numbers $\alpha_i$ such that the following conditions hold
\begin{enumerate}[(i)]
    \item $m_1\alpha_1+\ldots +m_j\alpha_j < \frac{\alpha}{n+1}$ (resp. $\leq$)
    \item $lct(\P^n,H_i)\geq \frac{1}{\alpha_i}$
\end{enumerate}
then $H$ is GIT stable (resp. semistable) (for the action of $SL(V)$ on $S^{\alpha}V^*$).
\end{prop}

\begin{proof}
Choose a one-parameter subgroup $\lambda: \C^{\times}\to SL(V)$. By \cite[Section 8]{singpairs} (see also Lemma \ref{toric} or \cite[Corollary 4.2]{azdegd}) we have:
\begin{align*}
\frac{\omega(H,\lambda)}{A_{\lambda}} &=\sum_{i=1}^j \frac{m_i\cdot \omega(H_i,\lambda)}{A_{\lambda}}\\
&\leq \sum_{i=1}^j\frac{m_i}{lct(\P^n,H_i)} \\
&\leq m_1\alpha_1+\ldots +m_j\alpha_j\\
&< \frac{\alpha}{n+1}\quad (\text{resp.}\,\leq)
\end{align*}
\end{proof}

Furthermore, we also have:

\begin{prop}
Let $H_f=H_{f_1}+\ldots + H_{f_{k+1}}$ be a reducible hypersurface of degree $d(k+1)$ in $\P^n$, where each $H_{f_i}$ has degree $d$. If all the $H_{f_i}$ are GIT semistable, then $H$ is GIT-semistable. Furthermore, if one of the $H_{f_i}$ is GIT stable, then $H$ is GIT-stable.
\label{components}
\end{prop}

\begin{proof}
Again, this follows from the fact that for any one-parameter subgroup $\lambda$ of $SL(V)$ we have $\omega(H,\lambda)=\omega(f_1,\lambda)+\ldots+\omega(f_{k+1},\lambda)$.

If $H$ is unstable (resp. non-stable), then we can find a one-parameter subgroup $\lambda: \C^{\times}\to SL(V)$ such that
\[
\frac{d(k+1)}{n+1} <\,(\text{resp.}\, \leq)\, \frac{\omega(f,\lambda)}{A_{\lambda}}=\sum_{j=1}^{k+1}\frac{\omega(f_j,\lambda)}{A_{\lambda}}
\]
If all the $H_{f_j}$ are semistable (resp. one of $H_{f_j}$ is further stable), then for this $\lambda$ we have $\frac{\omega(f_j,\lambda)}{A_{\lambda}}\leq \frac{d}{n+1}$ (resp. $<$) for all $j=1,\ldots,k+1$, which contradicts the previous inequality.
\end{proof}

\begin{rmk}
   Note that an analogue of Proposition \ref{components} holds for Chow stability, as observed by S.~Okawa in \cite{Okawa}. More precisely, by \cite[Proposition B.1]{Okawa} if two cycles of the same dimension in some projecive space are Chow semistable, then their sum is also Chow semistable; and, moreover, if one of the cycles is Chow stable, then so is their sum.
\end{rmk}

\section{Stability of pencils of plane curves}
\label{pencilsCurves}

For pencils of plane curves of degree $d$, Theorem \ref{mainTHM} becomes:

\begin{thm}
A pencil $\calP\in \mathscr{X}_{1,d,2}$ is  GIT stable (resp. semistable) if and only if for any two (distinct) curves $C_f$ and $C_f\in \calP$ the curve $C_f+C_g$ of degree $2d$ is GIT stable (resp. semistable).
\label{pencils}
\end{thm}

In this Section we will show how this criterion can be used to recover the results in \cite{stab} and we will also provide a complete description of the stability of so-called Halphen pencil of index $m$, which are pencils of plane curves of degree $3m$ that have exactly nine base points, each of multiplicity $m$.

\subsection{Pencils of plane cubics revisited}
\label{cubics}

The work of Miranda in \cite{stab} describes the GIT stability of pencils of plane cubics  (with smooth members) in terms of the types of singular fibers appearing in the corresponding rational elliptic surfaces, which is obtained by blowing up the plane at the nine base points of the pencil. Using Kodaira's notation for the singular fibers, he proves:

\begin{teo}[\cite{stab}]
A pencil of plane cubics $\calP$ is stable if and only if $\calP$ contains a smooth member and every fiber of the corresponding rational elliptic surface $X_{\calP}$ is reduced. Moreover, if $\calP$ contains a smooth member, then $\calP$ is semistable if and only if $X_{\calP}$ does not contain a fiber of type $II^*,III^*$ or $IV^*$.
\end{teo}

We will now describe how we can use Theorem \ref{pencils} to recover Miranda's result. First, observe that by \cite{shah} and Theorem \ref{pencils} we know a pencil of plane cubics $\calP$ is stable if and only if for any two (distinct) cubics $C_f,C_g$ in $\calP$ the sextic curve $C_f+C_g$ satisfies the following conditions:
\begin{enumerate}[(i)]
    \item it does not contain a multiple line as a component,
    \item it does not have consecutive triple points, and
    \item it does not have a point with multiplicity $\geq 4$
\end{enumerate}

Therefore, a pencil of plane cubics $\calP$ is stable if and only if the following conditions hold:
\begin{enumerate}[(i')]
\item $\calP$ contains a smooth member,
    \item any curve in $\calP$ is reduced, and
    \item at a base point of $\calP$ any curve in $\calP$ is either smooth or it has at worst a node as singularity.
\end{enumerate}

It is routine to check that conditions $(ii')$ and $(iii')$ hold if and only if every fiber of $X_{\calP}$ is reduced.

Finally, by \cite{shah} and Theorem \ref{pencils}, $\calP$ is unstable if and only if we can find two cubics $C_f,C_g$ in $\calP$ such that the sextic curve $C_f+C_g$ satisfies one of the following conditions:

\begin{enumerate}[(a)]
    \item it has a line as a component and it has a triple point on that line which remains a triple point with a threefold tangent under a blow-up, or
    \item it has a quadruple point which has a threefold or a fourfold tangent, or
    \item it has a singular point of multiplicity $\geq 5$
\end{enumerate}

Or, equivalently, $\calP$ is semistable if and only if for any two cubics $C_f,C_g$ in $\calP$ the sextic curve $C_f+C_g$ either has at worst ADE singularities or it  lies in the open orbit of a sextic containing a double conic or a triple conic in its support.  

Assuming $\calP$ has a smooth member, we see that one of the conditions $(a),(b)$ or $(c)$ holds if and only if, up to relabeling, either:

\begin{enumerate}[(a')]
    \item $C_f$ is a triple line and $C_g$ is arbitrary, or
    \item $C_f$ is the union of a double line and another line and $C_g$ is tangent to the double line at the intersection point (of the two lines), or
    \item $C_f$ is the union of a double line and another line and $C_g$ intersects the double line at a single point
\end{enumerate}

Now, if one of the conditions $(a'),(b')$ or $(c')$ holds then the fiber of $X_{\calP}$ corresponding to $C_f$ will have a component of multiplicity three, hence it must be of type $II^*,III^*$ or $IV^*$. Conversely, if $X_{\calP}$ contains a fiber of type $II^*,III^*$ or $IV^*$, then by \cite{azconstr}*{Proposition 4.2}  we know $\mathcal{P}$  is the pencil $\lambda C_f +\mu C_g =0$, where $C_f$ consists of either  a triple line or a double line and another line. This is the content of \cite{stab}*{Lemma 6.4}, which is the hardest step in the proof of Theorem 6.1 in \cite{stab}. Moreover, if $C_f$ consists of a double line and another line, then $C_f$ and $C_g$ must intersect as in $(b')$ or $(c')$, otherwise blowing-up the nine base points of $\mathcal{P}$ would not yield a fiber with a component with multiplicity three.

\subsection{Halphen pencils}
\label{halphen}

We now provide a complete description of the stability of certain pencils of plane curves of degree $3m$ which are called Halphen pencils of index $m$. First studied by the french mathematician Georges Henri Halphen in \cite{halphen}, such pencils are characterized by the property that they have exactly nine base points (possibly infinitely near), each of multiplicity $m$, and their general member is an integral curve. For details, see \cite{enriques}. Note that the case $m=1$ consists of exactly pencils of plane cubics with a smooth member. Moreover, the case $m=2$ is also well understood, and they have been considered by the second author in \cite{azconstr} and \cite{azstab}. Here we address all other cases.

Given a Halphen pencil of index $m$, say $\calP$, taking the minimal resolution of its base points, one obtains a so-called rational elliptic surface of index $m$. That is, a (smooth and projective) rational surface $Y$ that comes equipped with a genus-one fibration $Y\to \P^1$ given by $|-mK_{Y}|$ and which has exactly one multiple fiber of multiplicity $m$. As in \cite{stab} and \cite{azstab}, our strategy consists in exploring the geometry of $Y$ to describe the stability of $\calP$.

In what follows, we will prove the following:

\begin{thm}
Let $\calP$ be a Halphen pencil of index $m\ge1$ and let $Y$ denote the corresponding rational elliptic surface. If $lct(Y,F)>\frac{1}{2m}$ (resp. $\geq$) for any fiber $F$, then $\calP$ is GIT stable (resp. semistable).
\label{mainHalphen}
\end{thm}

Some direct consequences of Theorem \ref{mainHalphen} are:

\begin{cor}
Any Halphen pencil of index $m>3$ is GIT stable. Moreover, a Halphen pencil of index $m=3$ is always GIT semistable, and it is stable whenever the corresponding rational elliptic surface does not contain a fiber of type $II^*$.
\label{corHalphen1}
\end{cor}

\begin{cor}
Any Halphen pencil of index $m=2$ is GIT stable whenever the corresponding rational elliptic surface $Y$ does not contain a fiber of type $II^*$ or $III^*$. Moreover, if $Y$ does have a fiber of type $III^*$, then $\calP$ is always semistable (cf. \cite{azstab}).
\label{corHalphen2}
\end{cor}

\begin{rmk}
Note that the number $lct(Y,F)$ is given by the table below, depending on the type of $F$: 
\begin{table}[H]
\centering
%\caption{The log canonical threshold of the pair $(Y,F)$}
\label{lct}
\begin{tabular}{c|c||c|c}
\boldmath{$lct(Y,F)$} & \bf{Type of \boldmath{$F$}}&\boldmath{$lct(Y,F)$} & \bf{Type of \boldmath{$F$}}\\
\hline
$1/m$ & $_mI_n$ & $1/2$ & $I_n^*$\\
$5/6$ & $II$ & $1/6$ & $II^*$ \\
$3/4$ & $III$ & $1/4$ & $III^*$\\
$2/3$ & $IV$ & $1/3$ & $IV^*$
\end{tabular}
\end{table}

Moreover, the condition $lct(Y,F)>\frac{1}{2m}$ is known to be equivalent to the notion of uniform adiabatic K-stability introduced in \cite{Hat}. In particular, Theorem \ref{mainHalphen} implies GIT-stability of Halphen pencils is closely related to the existence of K\"{a}hler metrics with constant scalar curvature on the corresponding rational elliptic surfaces.
\label{lct}
\end{rmk}

Later, we will further prove Theorem \ref{mainHalphen2}, which says the converse of Theorem \ref{mainHalphen} also holds except for a few Halphen pencils of index $m=2$ and of index $m=3$. But first we prove Theorem \ref{mainHalphen}. In our proof we will need the following:

\begin{lema}\label{prop.crepant}
Let $\mathcal{P}$ be a pencil of plane curves of degree $m\cdot s$ with  $s^2$ base points (possibly infinitely near), each of multiplicity $m$. Suppose that $\mathcal{P}$ contains at least one integral curve. Let $\pi:Y\to \mathbb{P}^2$ denote the $s$-fold blow-up which resolves $\calP$ and let $q:Y\to \mathbb{P}^1$ be the morphism induced by $\mathcal{P}$. 

Then there exists a curve $C_f\in \calP$ such that for any other curve $C_g\in\mathcal{P}$, we have that
\[
K_Y+\frac{1}{2m}(F_f+F_g)=\pi^*\left(K_{\mathbb{P}^2}+\frac{1}{2m}(C_f+C_g)\right).
\]
where $F_f$ and $F_g$ are the corresponding fibers of $q$
\end{lema}

\begin{proof}
We choose $C_f$ to be a curve in $\calP$ corresponding to a smooth fiber of $q$. Given any other curve $C_g$ in $\calP$, let $F_g$ be the fiber of $q$ corresponding to $C_g$ and let $\mathrm{D}$ be the fixed part of $\pi^*\mathcal{P}$. Then, it is routine to check $m \cdot K_{Y/\mathbb{P}^2}=\mathrm{D}$. In particular,

%Here, we asserts that $mK_{Y/\mathbb{P}^2}=\mathrm{Bs}$. here, we understand $K_{Y/\mathbb{P}^2}$ is the discrepancy. Indeed, we can prove this by the following inductive argument. Let $\pi_j:Y_j\to Y_{j-1}$ be the one point blow ups where $Y_j$ is $j$ points of the $d^2$ base points. Here, we assume that $Y_0=\mathbb{P}^2$. Let $p_j:Y_j\to Y_0$ be the canonical birational morphism and $\mathrm{Bs}_j$ be the fixed part of $p_j^*\mathcal{P}$. If $mK_{Y_{j-1}/\mathbb{P}^2}=\mathrm{Bs}_{j-1}$ holds and $E_j$ is the exceptional curve of $\pi_j$, then
% \[
% mK_{Y_j/\mathbb{P}^2}=m\pi_j^*K_{Y_{j-1}/\mathbb{P}^2}+mE_j=\pi_j^*\mathrm{Bs}_{j-1}+mE_j=\mathrm{Bs}_j
 %\]
 %by the assumption. Thus, we have $mK_{Y/\mathbb{P}^2}=\mathrm{Bs}$ and hence

  \begin{align*}
  \pi^*\left(K_{\mathbb{P}^2}+\frac{1}{2m}(C_f+C_g)\right)&=K_Y-K_{Y/\mathbb{P}^2}+\frac{1}{2m}(F_f+F_g+2\mathrm{D})\\
  &=K_Y+\frac{1}{2m}(F_f+F_g).
  \end{align*}
  \end{proof}

The argument is then as follows:

\begin{proof}[Proof of Theorem \ref{mainHalphen}]

We first observe Halphen pencils of any index $m$ satisfy the assumptions of Lemma \ref{prop.crepant}. For we can choose $C_f$ to be any curve corresponding to a smooth fiber of the corresponding genus-one fibration. Now, choose any other curve $C_g\in \calP$. Lemma \ref{prop.crepant} implies that 
\[
lct(Y,F_f+F_g)>\frac{1}{2m}\,(\text{resp.}\,\geq)\,\iff lct(\mathbb{P}^2,C_f+C_g)>\frac{1}{2m} \,(\text{resp.}\,\geq)
\]

Note that $lct(Y,F_f)=1$. Thus, $C_f+C_g$ is stable (resp. semistable). Now, because $C_g$ was arbitrary, it follows from Corollary \ref{keycor} that $\mathcal{P}$ is GIT stable (resp. semistable).
\end{proof}

Next, in view of Corollaries \ref{corHalphen1} and \ref{corHalphen2}, we further prove:

\begin{thm}
Let $\calP$ be a Halphen pencil of index $m$ and assume $\calP$ is not given by \cite[Examples 7.46, 7.47, 7.55]{azconstr} nor Example \ref{exc.case3} below. If $lct(Y,F)<\frac{1}{2m}$ (resp., $\le\frac{1}{2m}$) for some fiber $F$, then $\calP$ is GIT unstable (resp., GIT non-stable).
    \label{mainHalphen2}
\end{thm}

To prove Theorem \ref{mainHalphen2} we first observe that, in view of Lemma \ref{prop.crepant} and Corollary \ref{hypersurftoric}, if there exists a toric divisor $E\subset Y$ over $\mathbb{P}^2$ (see Definition \ref{DefToricDiv}) and a fiber $F$ of $Y$ such that $\frac{A_{\mathbb{P}^2}(E)}{\mathrm{ord}_E(F)}\le\frac{1}{2m}$ (resp. $<$), then the Halphen pencil $\mathcal{P}$ is GIT non-stable (resp. unstable). We call such toric divisors destabilizing and we say $E$ makes $\calP$ non-stable (resp. unstable). 

Furthermore, we also observe the multiple cubic $mC$ in $\calP$ is such that $lct(\P^2,mC)=\frac{1}{m}$ \cite[Proposition 4.9]{azconstr} and that the following two Lemmas hold: 

\begin{lema}[{\cite[Theorem 1.1]{azconstr}}]
If $m>1$ and $F_f$ is a reduced fiber of $Y$, then the corresponding curve $C_f$ is reduced and we have that 
\[
\frac{1}{m}<lct(\mathbb{P}^2,C_f)\leq lct(Y,F_f).
\]
\label{ineq}
\end{lema}

\begin{lema}
If $m>1$ and $Y$ contains a fiber $F$ such that $lct(Y,F)\le\frac{1}{2m}$ but the corresponding plane curve $C_g\in\mathcal{P}$ is GIT semistable, then $\mathcal{P}$ is GIT stable.
\label{sscurve}
\end{lema}

\begin{proof}
First, note that $F$ is a fiber of type $II^*$ or $III^*$ by Remark \ref{lct}. Moreover, since the Picard number of $Y$ is $10$, $Y$ has at most one such fiber. Now, take $C_f\in\mathcal{P}$ corresponding to a smooth fiber $F_f$. By Corollary \ref{keycor}, it suffices to show that for any other curve $C_g\in\mathcal{P}$, $C_f+C_g$ is GIT stable. If $C_g$ does not correspond to the fiber $F$, then we can argue as in the proof of Theorem \ref{mainHalphen} to conclude $C_f+C_g$ is stable. 
Otherwise, $F=F_g$ and since $C_g$ is GIT semistable by assumption, it follows from Lemma \ref{ineq}, Corollary \ref{hypersurftoric} and Proposition \ref{components} that $C_f+C_g$ is GIT stable.
\end{proof}

Then we can use the following argument to prove Theorem \ref{mainHalphen2}:
\begin{proof}[Proof of Theorem \ref{mainHalphen2}]
The case $m=2$ has already been dealt with in \cite{azconstr}. Here we observe that we can always find destabilizing toric divisors for almost all Halphen pencils such that $Y$ contains a fiber $F$ of type $II^*$ or $III^*$ (see \cite{stab} and \cite[\S7]{azconstr}). The only exceptions are those pencils for which the curve corresponding to $F$ is semistable, which must be as in one of the Examples 7.46, 7.47, 7.55 in \cite{azconstr}. But for these exceptional cases we can then refer to Lemma \ref{sscurve}.

Therefore, it suffices to consider the case when $\mathcal{P}$ is a Halphen pencil of index $m=3$ which is not as in Example \ref{exc.case3} and $Y$ has a fiber $F$ of type $II^*$. We claim that for such pencils there will always be a toric divisor computing $lct(Y,F)$, hence making $\calP$ non-stable. More precisely, we can show that there exists an exceptional divisor on $Y$ over a base point of $\calP$ which is toric in the sense of Definition \ref{DefToricDiv}.

To prove the claim, the key ingredients we will use are: \cite[Lemma 4.1]{azconstr}; the fact that $(-1)$-curves are trisections of $Y$; and Remark \ref{fac1}.

Our argument is as follows: First, recall that a fiber of type $II^*$ has the following dual graph:
$$
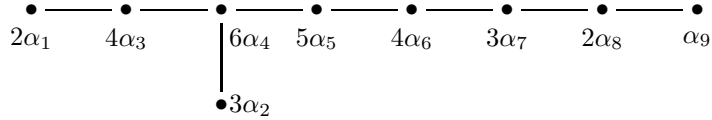
\xygraph{
    \bullet ([]!{+(0,-.3)} {2\alpha_1}) - [r]
    \bullet ([]!{+(0,-.3)} {4\alpha_3}) - [r]
    \bullet ([]!{+(.3,-.3)} {6\alpha_4}) (
        - [d] \bullet ([]!{+(.3,0)} {3\alpha_2}),
        - [r] \bullet ([]!{+(0,-.3)} {5\alpha_5})
        - [r] \bullet ([]!{+(0,-.3)} {4\alpha_6})
        - [r] \bullet ([]!{+(0,-.3)} {3\alpha_7})
        - [r] \bullet ([]!{+(0,-.3)} {2\alpha_8})
         - [r] \bullet ([]!{+(0,-.3)} {\alpha_9})
)}$$

We will now show $\alpha_4$ is always toric (Definition \ref{DefToricDiv}). Note that if $\alpha_4$ is non-exceptional then $\alpha_4$ is automatically toric so, it suffices to assume $\alpha_4$ is an exceptional curve. Moreover, since $\calP$ is a pencil of plane curves of degree nine there are four cases to consider:

\begin{enumerate}[{Case} 1]
    \item $\alpha_2$ is the strict transform of a cubic: This case cannot happen. In fact there exists a unique triple cubic in $\calP$ and that corresponds to the unique multiple fiber (see e.g. \cite[Corollary 5.6.3]{enriques}). Alternatively, note that if we blow down all the $\alpha_i$ (except $\alpha_2$) and the $(-1)$-curve, then the proper image of $\alpha_2$ in $\mathbb{P}^2$ would be a cubic curve that  has self-intersection smaller than five, which is an absurd.
    
    \item $\alpha_2$ is the strict transform of a conic: Again, this case cannot happen. By contradiction, assume $\alpha_2$ does come from a conic. If $\alpha_1$ is non-exceptional, then $\alpha_1$ and $\alpha_9$ must be the strict transform of lines in $\mathbb{P}^2$. But then if we blow down $\alpha_j$ for $3\le j\le8$ and a $(-1)$-curve intersecting the $\alpha_j$'s, then the proper image of $\alpha_1$ or $\alpha_9$ would have self-intersection number greater than $1$, which is impossible. Now, if $\alpha_1$ is exceptional, then there are three possibilities:

    \begin{enumerate}[(a)]
        \item $\alpha_7$ is non-exceptional
        \item $\alpha_8$ and $\alpha_9$ are non-exceptional and given by lines but $\alpha_7$ is exceptional
        \item all $\alpha_j$ are exceptional except for $\alpha_2$ and $\alpha_9$
    \end{enumerate}
and we claim none of these cases can actually occur.
    
    First, arguing as in \cite[Section 5]{azconstr}, we can show the curve corresponding to the type $II^*$ fiber cannot consist of the union of a triple conic and a triple line. In particular, $\alpha_7$ is exceptional. 
    
    Second, if we assume $\alpha_8$ and $\alpha_9$ are non-exceptional and given by lines, then there is a $(-1)$-curve intersecting either $\alpha_7$, or both $\alpha_1$ and $\alpha_9$. If the former holds, then the strict transform of $\alpha_8$ has self-intersection number at least five. Otherwise, the strict transform of $\alpha_9$ has self-intersection number at least five. In any case we reach a contradiction.

Finally, assume $\alpha_9$ is given by a cubic. By \cite[Lemma 4.1]{azconstr}, there exists a $(-1)$-curve $E_1$ intersecting $\alpha_1$ or $\alpha_8$. Let $E_2$ be the other $(-1)$-curve introduced by the nine-fold blow-up. If $E_1$ intersects $\alpha_8$, then the strict transform of $\alpha_2$ has self-intersection number $1$ when we contract $E_1$ and the exceptional $\alpha_j$'s, which implies the proper image of $\alpha_2$ in $\mathbb{P}^2$ (after further contracting $E_2$) has self-intersection number less than three, a contradiction. Similarly, if $E_1$ intersects $\alpha_1$, we can argue  the proper image of $\alpha_9$ in $\mathbb{P}^2$ has self-intersection number seven or ten, which is again a contradiction.

    \item $\alpha_2$ is the strict transform of a line: Assume $\alpha_2$ comes from $3L$. In this case, we see that $\alpha_4$ is always an exceptional divisor of a blow-up at a point $p'$ infinitely near to a point $p\in L$ as Remark \ref{fac1}. In particular, $\alpha_4$ is toric. Note that Example \ref{exhal} is included in this case.
    
    \item $\alpha_2$ is exceptional: By \cite[Lemma 4.1]{azconstr}, either $\alpha_3$ or $\alpha_5$ must be non-exceptional. If either of them comes from a line, then $\alpha_4$ is toric. If $\alpha_3$ comes from a conic and $\alpha_5$ is non-exceptional, then $\alpha_9$ must come from a line and we are in the situation of Example \ref{exc.case3} (see also Remarks \ref{line4conic} and \ref{line4conic2}).
\end{enumerate}

This shows $\alpha_4$ is always toric, except for a Halphen pencil as in Example \ref{exc.case3},  as asserted.
\end{proof}

Lastly, we end this section with two explicit examples that show there are, indeed, stable and non-stable Halphen pencils of index $m=3$ yielding a fiber of type $II^*$.    

\begin{exe}\label{exhal}
Consider the cubic $C$ given by $
z^2y+x(y^2+xz)=0$. Let $Q$ be the conic $y^2+xz=0$ and let $L$ be the line $y=0$. Then the pencil generated by $C_f=2Q+5L$ and $C_g=3C$ is a Halphen pencil  of index three, say $\calP$, such that the corresponding rational elliptic surface contains a fiber of type $II^*$ and such pencil is GIT non-stable. 

To see why, let $\lambda$ be the one-parameter subgroup determined (in these coordinates) by $a_0=1,a_1=0$ and $a_2=-1$. Then, arguing as in \cite[Lemma 3.3]{lct}, we compute $\omega(fg,\lambda)=18$ and it follows from Lemma \ref{keylemma2} that
\[
6= \frac{\omega(fg,\lambda)}{3}=\frac{\omega(f,\lambda)+\omega(g,\lambda)}{3}\leq \frac{\omega(\calP,\lambda)}{3}=\frac{\omega(\calP,\lambda)}{(a_0-a_2)-(a_1-a_2)}
\]
Thus $\calP$ is non-stable by Proposition \ref{HM}.
\end{exe}

\begin{exe}
Consider the cubic $C$ given by $
(y^2+xz)(\alpha y+z)+\beta yx^2=0$, for some $\alpha,\beta\neq 0$. Let $Q$ be the conic $y^2+xz=0$ and let $L$ be the line $y=0$. Then the pencil generated by $C_f=4Q+L$ and $C_g=3C$ is a Halphen pencil  of index three such that the corresponding rational elliptic surface contains a fiber of type $II^*$ and such pencil is GIT stable by Lemma \ref{sscurve}. 
\label{exc.case3}
\end{exe}

\begin{rmk}
We observe that arguing as in \cite[Section 5]{azconstr}, we can show that if a Halphen pencil of index $m=3$ yields a fiber $F$ of type $II^*$ and the curve corresponding to $F$ consists in the union of a conic taken with multiplicity four and a line, then the conic and the line must intersect at two distinct points, say $P_1$ and $P_2$. Moreover, up to relabeling, the unique triple cubic in the pencil, say $3C$, is such that $C$ intersects $Q$ (resp. $L$) with multiplicity five (resp. one) at $P_1$ and it intersects $Q$ (resp. $L$) with multiplicity one (resp. two) at $P_2$. In particular, to obtain the corresponding rational elliptic surface $Y$ we must blow up the plane seven times at $P_1$ and two times at $P_2$.
\label{line4conic}
\end{rmk}

\begin{rmk}
Note that if $Q,C$ and $L$ are as in Remark \ref{line4conic}, then up to a change of coordinates we may assume $P_1=(0:0:1)$ $P_2=(1:0:0)$, $Q$ is the conic given by $y^2+xz=0$ and $L$ is the line given by $y=0$. This further implies that, $C$ can be taken to be a cubic as in Example \ref{exc.case3}.

In fact, up to scaling, the above implies $C$ is given by
\[
xz^2+\alpha y^3+ \beta x^2y +a xy^2+b xyz+cy^2z+d yz^2=0 \qquad \text{with\,\,$\alpha\neq 0$}
\]
Now, we further know the line $x=0$ is tangent to $C$ at $P_1$ (with multiplicity two, since $P_1$ cannot be a flex). Thus, we may further assume $d=0$ and $c=1$. We can then rewrite the defining polynomial of $C$ as
\[
(y^2+xz)(\alpha y+z)+xy(\beta x+ay+(b-\alpha)z)=0
\]
Thus,
\[
I_{P_1}(C,Q)=I_{P_1}(Q,L)+I_{P_1}(Q,L')+I_{P_1}(Q,L'')=3+I_{P_1}(Q,L'')
\]
where $L':x=0$ and $L'':\beta x+ay+(b-\alpha)z=0$. And since the intersection multiplicity $I_{P_1}(C,Q)$ is equal to five, we must have $I_{P_1}(Q,L'')=2$, which tells us $L''= L'$, hence $a=0$ and $b=\alpha$. In other words, $C$ is as in Example \ref{exc.case3}, as claimed.  
\label{line4conic2}
\end{rmk}

\section{Stability of nets of conics}
\label{conics}

In this Section we will now describe how one can apply the criterion in Theorem \ref{mainTHM} in order to obtain an alternative geometric description of the stability of nets of conics. The stability of linear systems of quadrics in general has been explored in \cite{quadrics}, \cite{quadricsMir}, \cite{wall} and \cite{genus5}. Here we provide an explicit description of the stability of nets of conics using Theorem \ref{thmINTRO}, and we also explain how our description agrees with the description in \cite{wall} by C.T.C Wall. But we point out that the connection between the two approaches can also be somewhat easily read off from \cite[Table 1]{translation}.

First, we recall Theorem \ref{thmINTRO} together with the work of Shah in  \cite{shah} tell us a net of conics $\calN$ is stable if and only if for any three cubics $C_f,C_g$ and $C_h$ in $\calN$, the sextic curve $C_f+C_g+C_h$ satisfies conditions $(i),(ii)$ and $(iii)$ from Section \ref{cubics}.

Therefore, we obtain: 

\begin{prop}
A net of conics $\calN$ is GIT stable if and only if the following conditions hold:
\begin{enumerate}[(i'')]
\item $\calN$ does not contain a double line
       \item at a base point of $\calN$ any conic in $\calN$ is smooth (in particular, the base locus is zero-dimensional and $\calN$ contains a smooth member)
    \item at a base point of $\calN$ no three conics in $\calN$ can be mutually tangent 
    \item every pencil contained in $\calN$ has a smooth member
\end{enumerate} 
\label{conicsStab}
\end{prop}

On the other hand, we can prove the following two lemmas:

\begin{lema}
    If a net of conics $\calN$ has a base point $p$, then there exists a conic in $\calN$ which is singular at $p$.
    \label{nobasepts}
\end{lema}

\begin{proof}
    Choose generators $C_f,C_g$ and $C_h$ for $\calN$ and choose coordinates $(x_0:x_1:x_2)$ in $\P^2$. If $p=(a:b:c)$ is a base point of $\calN$, then $(a,b,c)$ is a non-trivial element in the kernel of the matrix
    \[
    A\doteq \begin{pmatrix} \frac{\partial C_f}{\partial x_0 }  & \frac{\partial C_f}{\partial x_1 }  & \frac{\partial C_f}{\partial x_2 }  \\ & & \\ \frac{\partial C_g}{\partial x_0 }  & \frac{\partial C_g}{\partial x_1 }  & \frac{\partial C_g}{\partial x_2 }  \\ & & \\ \frac{\partial C_h}{\partial x_0 }  & \frac{\partial C_h}{\partial x_1 }  & \frac{\partial C_h}{\partial x_2 } \end{pmatrix}
    \]
    In particular, $\det A=\det A^{t}=0$, which implies we can find $(\lambda:\mu:\nu)\in \P^2$ such that 

    \[
    \lambda \frac{\partial C_f}{\partial x_0 }  + \mu \frac{\partial C_g}{\partial x_0 }  + \nu \frac{\partial C_h}{\partial x_0 } = 0 
    \]
    \[
    \lambda \frac{\partial C_f}{\partial x_1 }  + \mu \frac{\partial C_g}{\partial x_1 }  + \nu \frac{\partial C_h}{\partial x_1 } = 0 
    \]
    \[
    \lambda \frac{\partial C_f}{\partial x_2 }  + \mu \frac{\partial C_g}{\partial x_2 }  + \nu \frac{\partial C_h}{\partial x_2 } = 0 
    \]
    
    And the above equations tell us the curve $\lambda C_f + \mu C_g + \nu C_h$ is singular at $p$.
\end{proof}

\begin{lema}
If a net of conics $\calN$ does not contain a double line and it does not have a base point, then every pencil contained in $\calN$ has a smooth member.
\end{lema}

\begin{proof}
    We first observe that if a pencil of conics $\calP$ does not contain a smooth member then one of the following conditions hold (up to change of coordinates in $\P^2$):
    \begin{enumerate}[(I)]
        \item the base locus of $\calP$ is one-dimensional; or
        \item $\calP$ is generated by  $C_f: x_0^2=0$ and $C_g: x_1^2=0$; or
        \item $\calP$ is generated by 
                 $C_f: x_0x_1=0$ and $C_g: (ax_0+bx_1)(cx_0+dx_1)=0$ for some $a,b,c,d \in \C$ not all zero.
        
           \end{enumerate}
    
    Now, since $\calN$ does not contain a double line and it does not have a base point, it is clear that $\calN$ cannot contain a pencil as in (I) or (II). Thus it suffices to check $\calN$ cannot contain a pencil as in (III) either. In fact, a pencil $\calP$ as in (III) always contains a double line, namely the curve $\lambda C_f + C_g$, where  $\lambda=2\sqrt{abcd}-(ac+bd)$.

\end{proof}

Therefore, Proposition \ref{conicsStab} becomes:

\begin{prop}
A net of conics $\calN$ is GIT stable if and only if $\calN$ does not contain a double line and it does not have a base point.
\label{conicsStab2}
\end{prop}

Now, by \cite{shah} and Theorem \ref{pencils}, we further know $\calN$ is unstable if and only if we can find three conics $C_f,C_g$ and $C_h$ in $\calN$ such that the sextic curve $C_f+C_g+C_h$ satisfies one of the conditions $(a),(b)$ or $(c)$ from Section \ref{cubics}. In particular,

\begin{prop}
A net of conics $\calN$ is unstable if and only if we can find three conics $C_f,C_g$ and $C_h$ in $\calN$ such that one of the following (non-mutually-exclusive) conditions holds:
\begin{enumerate}[(a'')]
    \item the base locus of $\calN$ is one-dimensional
          \item 
          \begin{itemize}
              \item $C_f=2L$,
              \item $C_g$ is tangent to $C_f$ at a base point of $\calN$, and
              \item $C_h$ is arbitrary
          \end{itemize}
          \item $C_f=L+L'$ and either:
          \begin{itemize}
              \item  $C_g=L+L''$ and $C_h$ is smooth and tangent to $C_f$ at the intersection point $L\cap L'$; or
              \item $C_g$ and $C_h$ are smooth and $C_f,C_g$ and $C_h$ are mutually tangent at $L\cap L'$. 
          \end{itemize}

    \item $C_f$ and $C_g$ intersect at a base point of $\calN$ which is a singular point of both $C_f$ and $C_g$, and $C_h$ is arbitrary
\end{enumerate}
\label{conics Unstab}
\end{prop}

And as a consequence we obtain:
\begin{cor}
    If a net of conics has no base points, then it is GIT semistable.
\end{cor}

Note further that if a net of conics $\calN$ contains a smooth member, then by Corollary \ref{mainCOR2} and \cite[1.12]{mumford} we also know that if $\calN$ is unstable, then we can find two conics $C_f,C_g \in \calN$ such that $C_f+C_g$ has a triple point. That is, up to relabeling, we can find a conic $C_f\in 
\calN$ which is singular and another conic $C_g\in \calN$ which intersects $C_f$ at a singular point of $C_f$.

Finally, we explain how Propositions \ref{conicsStab} (or \ref{conicsStab2}) and \ref{conics Unstab} above agree with the criteria in \cite{wall} which says:

\begin{teo}[\cite{wall}]
A net of conics is GIT stable (resp. semistable) if and only if the corresponding discriminant cubic curve is smooth (resp. has at worst nodes).
\end{teo}

Note that any conic $C_f \in \P^2$ can be described by an equation of the form
\[
x^{t}A_f x=\sum_{i=0}^2\sum_{j=0}^2 f_{ij} x_ix_j=0 
\]
where $A_f=(f_{ij})$ is a symmetric $3\times 3$ matrix. In particular, given a net of conics, say $\lambda C_f+\mu C_g+\nu C_h=0$, its discriminant is the ternary cubic form
\[
\Delta=\Delta(\lambda,\mu,\nu) \doteq \det(\lambda A_f+\mu A_g+\nu A_h).
\]

Note also that the singularities of $\Delta=0$ do not depend on the choice of generators $C_f,C_g$ and $C_h$.  

We can prove the following:

\begin{prop}
The plane cubic $\Delta=0$ is singular if and only if there exists a choice of generators $C_{f'},C_{g'}$ and $C_{h'}$ such that the sextic curve $C_{f'}+C_{g'}+C_{h'}$ is GIT non-stable.
\label{quadrics}
\end{prop}

\begin{proof}
Write $\Delta=\sum_{i,j}\delta_{ij}\lambda^i\mu^j\nu^{3-i-j}$ and assume the cubic $C:\Delta=0$ is singular. Then  there exists a choice of  generators $C_{f'},C_{g'}$ and $C_{h'}$ (equivalently, a choice of coordinates $(\lambda:\mu:\nu)$) such that $\delta_{00}=\delta_{01}=\delta_{10}=0$. That is, $C$ is singular at $(0:0:1)$.

In particular, $\det(A_{h'})=\delta_{00}=0 \iff C_{h'}$ is singular. Now, if $C_{h'}$ is a double line, then $C_{f'}+C_{g'}+C_{h'}$ is non-stable by \cite{shah}. And, similarly, if $C_{h'}$ is the union of two lines, then we claim the intersection point of the two lines is a base point of the net, which implies $C_{f'}+C_{g'}+C_{h'}$ contains a point of multiplicity $\geq 4$, hence it is non-stable by \cite{shah}. In fact, we can choose coordinates $(x_0:x_1:x_2)$ such that $C_{h'}$ is given by $x_0x_1=0$. In particular,
\[
\delta_{01}=-g'_{22} \quad \text{and} \quad \delta_{10}=-f'_{22} 
\]
which implies $(0:0:1)\in C_{f'}\cap C_{g'}$.

Conversely, assume there exists a choice of  generators $C_{f'},C_{g'}$ and $C_{h'}$ such that $C_{f'}+C_{g'}+C_{h'}$ is non-stable. Then, up to relabeling, one of the following conditions holds:
\begin{enumerate}[{Case} I]
    \item $C_{h'}$ is a double line
    \item $C_{h'}$ is the union of two lines that intersect at a base point of the net
     \item the three conics are mutually tangent at a base point of the net
    \item $C_{g'}$ and $C_{h'}$ form a pencil such that any member is singular.
\end{enumerate}

If either one of Cases I or II holds, then we can find coordinates $(x_0,x_1,x_2)$ in $\P^2$ such that $\delta_{00}=\det(A_{h'})=0$ and 
\[
\delta_{01}=-g'_{22}(h'_{01})^2 \quad \text{and} \quad \delta_{10}=-f'_{22}(h'_{01})^2 
\]
In Case I we can assume $C_h'$ is given by $x_0^2=0$, hence $h'_{01}=0$. And in Case II we can assume $C_h'$ is given by $x_0x_1=0$, which is singular at $(0:0:1)$, and since such point is a base point of the net, we must further have $f'_{22}=g'_{22}=0$. In any case, the conclusion is that $\delta_{00}=\delta_{01}=\delta_{10}=0$ and the cubic $C$ is singular at $(0:0:1)$.

Now, assume Case III holds. Since we can always find a singular member in the net, and we have already considered Cases I and II, we may choose coordinates $(x_0,x_1,x_2)$ in $\P^2$ such that $C_{h'}$ is given by $x_0x_1=0$ and the base point in question is $(0:1:0)$. The case when $C_{f'}$ or $C_{g'}$ is singular is included in Case IV. Thus, we may assume the curves $C_{f'}$ and $C_{g'}$ are such that $g'_{12}=g'_{11}=f'_{12}=f'_{11}=0$ but $f'_{22}\cdot g'_{22}\neq 0$. Then, $C$ is the singular cubic
\[
(f'_{22}\lambda +  g'_{22}\mu )(f'_{01}\lambda  +  g'_{01}\mu + \nu)^2=0
\]

Finally, if Case IV holds, then the cubic $C$ contains the line $\lambda=0$ and it is therefore singular as well.
\end{proof}

Similarly, by making convenient choices of coordinates $(x_0:x_1:x_2)$ in $\P^2$, we can also prove:

\begin{prop}
The plane cubic $\Delta=0$ has a singularity worse than a node if and only if there exists a choice of generators $C_{f'},C_{g'}$ and $C_{h'}$ such that the sextic curve $C_{f'}+C_{g'}+C_{h'}$ is GIT unstable.
\label{quadrics}
\end{prop}

In fact, up to a change of coordinates, we have the following correspondence (see also \cite[Table 1]{translation}):

\renewcommand{\arraystretch}{2}
\begin{table}[H]
\centering
\begin{tabular}{c|c | c}
The cubic $C$ & The net gen.  by $C_{f'},C_{g'}$ and $C_{h'}$ & $C_{f'}+C_{g'}+C_{h'}$  \\
\hline
$2\lambda^2\mu+\nu^3=0$ & $\lambda x_0x_1 + \mu x_2^2 + \nu(x_0^2+x_1x_2)=0$ & $x_0x_1x_2^2(x_0^2+x_1x_2)=0$ \\
$\nu(\nu\lambda+\mu^2)=0$ & $\lambda x_0^2 + \mu x_0x_1 + \nu(x_1+x_2)x_2=0$  & $x_0^3x_1x_2(x_1+x_2)=0$ \\
$\nu\mu^2=0$ & $\lambda x_0x_2 + \mu x_0x_1 + \nu x_2^2=0$ & $x_0^2x_1x_2^3=0$ \\
$\nu^2\mu=0$ & $\lambda x_1^2 + \mu x_2^2 +\nu x_0x_1=0$ & $x_0x_1^3x_2^2=0$ \\
 $\nu\mu^2=0$ & $\lambda x_0x_2 +\mu x_1x_2 + \nu(x_0^2+x_2^2)=0$ & $x_0x_1x_2^2(x_0^2+x_2^2)=0$ \\
 $\nu^3=0$ & $\lambda x_1^2 + \mu x_0x_1 + \nu (x_0^2+x_1x_2)=0$ & $x_0x_1^3(x_0^2+x_1x_2)=0$ \\
 $\Delta\equiv 0$ & $\lambda x_0^2 + \mu x_1^2 + \nu (x_0+x_1)^2=0$ & $x_0^2x_1^2(x_0+x_1)^2=0$ \\
 $\Delta\equiv 0$ & $\lambda x_0^2 + \mu x_0x_1 + \nu x_0x_2 =0$ & $x_0^4x_1x_2=0$
\end{tabular}
\end{table}

  \noindent \textbf{A remark.} Lastly, we end this section with a remark. Note that any net of conics $\calN$ can be thought of as divisor of bidegree $(1,2)$ in $\P^2\times\P^2$. It defines a Fano threefold $X$ which lies in the family \textnumero 2.24. In \cite[Section 2]{wall}, it is shown the stability of $X$ under the action of $PGL(3)\times PGL(3)$ agrees with the stability of $\calN$ under the action of $PGL(3)$. Here we further observe that the Jacobian criterion gives us the following: If $X$ is smooth, then either conditions $(ii'')$ and $(iii'')$ in Proposition \ref{conicsStab} hold, or $\calN$ does not have a base point.  In particular, combining this with Lemma \ref{nobasepts} and Propositions \ref{conicsStab}, \ref{conicsStab2}  and \ref{conics Unstab} we obtain:
  
  \begin{prop}
  Let $X$ be a smooth Fano threefold in the family \textnumero 2.24. Then $X$ is always GIT semistable, and $X$ is GIT strictly semistable whenever $\calN$ contains a double line.
  \end{prop}
  
  As a consequence, we can recover \cite[Corollary 4.7.8]{fanos} by \cite[Lemma 4.7.7]{fanos} and by the same argument of \cite[Theorem 3.4]{0SS} (see also the arguments in \cite[\S5.3]{Hat} and Appendix \hyperref[app]{A}). That is, we conclude $X$ is strictly K-semistable whenever $\calN$ contains a double line and it has no base points.

\section*{Appendix A: On the Chow stability of complete intersections}
\label{app}

In this appendix we extend Corollary \ref{chow} as follows:

 \begin{prop}\label{Sano's}
 Let $X\subset \mathbb{P}^n$ be a complete intersection defined by degree $d$ hypersurfaces $H_1,\cdots ,H_r$ for $r\le n$. If $X$ is Chow (semi)stable (see the definition in \cite{mumford}), then the linear system $\mathcal{L}$ generated by $H_1,\cdots,H_r$ is GIT (semi)stable.
 \end{prop}

In order to prove this, we first recall the following:

\begin{defi}
 Let $X$ be a projective variety on which $SL(n+1)$ acts. A line bundle $L$ on $X$ is called a $SL(n+1)$-linearized line bundle if $SL(n+1)$ acts on $\mathbf{L}=\mathrm{Spec}_X\left(\bigoplus_{m\ge0}L^{\otimes m}\right)$ such that the canonical projection $\mathbf{L}\to X$ is $SL(n+1)$-equivariant.
 
 Now, given $X$ as above and $L$ a $SL(n+1)$-linearized line bundle, we have that for any $x\in X$ and any one-parameter subgroup $\lambda$ of $SL(n+1)$, there exists a unique point $y\in X$ such that $y=\lim_{t\to0}\lambda(t)\cdot x$. Then the Hilbert-Mumford weight is defined by
 \[
 \mu^L(x,\lambda)=-\textrm{the weight of the action of }\lambda \textrm{ on }L\otimes k(y).
 \]
 where $k(y)$ is the residue field corresponding to $y$.
 
 %and we say that $\mathbb{G}_m$ acts on a vector space with the weight $\mu$ if $t\cdot w=t^\mu w$ for any $w\in W$.
\end{defi}

Then the Hilbert-Mumford criterion for GIT stability \cite[Theorem 2.1]{GIT} tells us that if $L$ is ample, then a closed point $x\in X$ is GIT stable (resp., semistable) if and only if  $\mu^L(x,\lambda)>0$ (resp., $\mu^L(x,\lambda)\ge0$) for any non-trivial one-parameter subgroup $\lambda$.

With this in mind, we can now prove  Proposition \ref{Sano's}.

 \begin{proof}[Proof of Proposition \ref{Sano's}]
 
 Suppose that the $H_i$ are represented by homogeneous polynomials $h_i$, each of degree $d$, and denote $\mathcal{L}$ by $(h_1,\cdots,h_r)$. As before, we consider $(h_1,\cdots,h_r)$ as a point in the Grassmannian variety $G$ parametrizing $r$-dimensional subspaces of $\Gamma(\mathbb{P}^n,\mathcal{O}(d))$. 
 
 Let $\mathcal{W}\subset G\times\Gamma(\mathbb{P}^n,\mathcal{O}(d))$ be the universal subspace associated with $G$ and let $\mathcal{X}\subset G\times\mathbb{P}^n$ be the closed subscheme defined by the ideal generated by the image of $\mathcal{W}\otimes\mathcal{O}(-d)\to\mathcal{O}_{G\times\mathbb{P}^n}$. Further, let $U\subset G$ be the locus of points $p$ such that the fiber $\mathcal{X}_p$ has dimension $n-r$. Then $U$ is Zariski open, $\mathcal{X}|_U\to U$ is flat and there exists a morphism $q:U\to C$ where $C$ is the Chow variety parametrizing $r$-codimensional cycles of degree $d^r$. This morphism is the composition of the natural map from $U$ to the Hilbert scheme induced by $\mathcal{X}|_U$ and the Hilbert--Chow morphism (cf.~\cite[\S5.4]{GIT}). Moreover, we view $C$ as a subvariety of $\mathbb{P}(W)$, where $W=\otimes^{r+1}(S^{d^r}(\mathbb{C}^{n+1}))$. For details, see e.g.~\cite[\S1.16]{mumford}.
 
 % where we write $W=\oplus_m W_{m}$ for the coordinate ring of the Grassmannian $G(r+1,n+1)$ (parametrizing $r$-dimensional subspaces of $\mathbb{P}^n$) in the Pl\"{u}cker embedding.\todo{Would you mind me writing ``we view $C\subset\mathbb{P}(\otimes^{r+1}S^{d^r}(\mathbb{C}^{n+1}))$. For details, see \cite[\S1.16]{mumford}" instead? This is much shorter. Aline: Sure! }

 %From the well-known facts that closed subschemes of $\mathbb{P}^n$ defined by $r$-sections are of codimension at most $r$ and that $G\ni p\mapsto\mathrm{dim}\,\mathcal{X}_p\in\mathbb{N}$ is upper semicontinuous (cf., \cite[II, Exer. 3.22 (d)]{Ha}), it follows that $U$ is Zariski open. For $p\in U$, let $f_1,\cdots,f_r$ define $\mathcal{X}_p$. Since $\mathbb{P}^n$ is Cohen-Macaulay, if $Z_i$ is the closed subscheme defined by $(f_1,\cdots,f_i)$, then $Z_i$ is a Cartier divisor of $Z_{i-1}$ defined by a section of $H^0(Z_{i-1},\mathcal{O}(d)|_{Z_{i-1}})$ inductively. Hence, $\mathcal{X}_p$ has the same Hilbert polynomial for $p\in U$ and, since $U$ is smooth, $\mathcal{X}|_U\to U$ is flat \cite[III, 9.9]{Ha}.
 
%Since $\mathcal{X}|_U\to U$ is flat, there exists a morphism $q:U\to C$ where $C$ is the Chow variety parametrizing $r$-codimensional cycles of degree $d^r$ (via the Hilbert-Chow morphism cf.~\cite[\S5.4]{GIT}. This morphism maps closed subschemes, points of the Hilbert scheme, to the corresponding cycles, the points of the Chow variety). 

Now, recall that the Chow stability of a cycle is defined to be the GIT stability of the corresponding point in $C$, with respect to the canonical   very ample $SL(n+1)$-linearized line bundle $M$ on $C$ which is the restriction of $\mathcal{O}_{\mathbb{P}(W)}(1)$ to $C$. 

%defined as follows (cf.~\cite[\S4.6, \S5.4]{GIT}). Recall that $C$ is a subvariety of $\mathbb{P}(W)$, where $$W=\Gamma\left((\mathbb{P}^n)^{\times n-r+1},\bigotimes_{i=1}^{n-r+1}p_i^*(\mathcal{O}_{\mathbb{P}^n}(d^r))\right).$$ Then $M$ is the restriction of $\mathcal{O}_{\mathbb{P}(W)}(1)$ to $C$. 

Let $L$ be an ample generator of $\mathrm{Pic}(G)$ and let $Y$ be the normalized graph of the rational map $r:G\dashrightarrow C$ induced by $q$. Note that $Y$ admits an $SL(n+1)$-left action compatible with $q$. Denote by $p_1:Y\to G$ and $p_2:Y\to C$ the projections. We claim that we have the following $SL(n+1)$-equivariant linear equivalence
\begin{equation}
p_2^*M\sim p_1^*L-bE
\label{linearequiv}
\end{equation}
where $E$ is a $p_1$-exceptional Cartier (effective) divisor.

In particular, given any $p\in U$ corresponding to $(h_1,\cdots, h_r)$ and any arbitrary non-trivial one-parameter subgroup $\lambda$ of $SL(n+1)$, if we let $\tilde{p}\in Y$ be the  point corresponding to $p$, then 
\begin{align*}
\mu^{L}(p,\lambda)&=\mu^{p_1^*L}(\tilde{p},\lambda)\\
&=\mu^{p_2^*M}(\tilde{p},\lambda)-b\mu^{\mathcal{O}_Y(-E)}(\tilde{p},\lambda)\\
&\ge\mu^{p_2^*M}(\tilde{p},\lambda)>0,\, (\textrm{resp., }\ge0)
\end{align*}
whenever $X$ is Chow (semi)stable \cite[Theorem 2.1]{GIT}. Here, we remark that for any $SL(n+1)$-stable effective Cartier divisor $E$ on $Y$, $\mu^{\mathcal{O}_Y(-E)}(y,\lambda)\le0$ for any closed point $y\in Y$ and any one-parameter subgroup $\lambda: \mathbb{C}^{\times}\to SL(n+1)$. 

Thus, it suffices to prove the claim. We first argue that we can always find a one-parameter subgroup $\lambda_0$ of $SL(n+1)$ and a point $p_0\in U$ such that the sign of $\mu^L(p_0,\lambda_0)$ coincides with the sign of  $\mu^M(q(p_0),\lambda_0)$. In fact we can exhibit $p_0$ and $\lambda_0$ such that $\lambda_0$ destabilizes both of $p_0$ and $q(p_0)$. In particular, up to replacing $M$ and $L$ by positive multiples, we may further assume that $\mu^{M}(q(p_0),\lambda_0)=\mu^L(p_0,\lambda_0)$. 

Recall that $SL(n+1)$ acts on $G$ as follows. Given $g\in SL(n+1)$, $g$ maps $(h_1,\cdots ,h_r)$ to $(g(h_1),\cdots,g(h_r))$, where
 \[
 g(h_i)(gx)=h_i(x) \qquad \forall\,\, x\in\mathbb{P}^n.
 \]
 Moreover, $q:U\to C$ is $SL(n+1)$-equivariant for this action. Consider $p_0\in U$ to be the point corresponding to the linear system $(x_0^d,x_1^d,\cdots,x_{r-1}^d)$ and let $\lambda_0$  be  a one-parameter subgroup of $SL(n+1)$ acting on the subspace in $\mathbb{P}^n$ spanned by $x_0,\ldots,x_{r-1}$ with weight 
 $-(n-r+1)$ and on the subspace spanned by $x_r,\ldots,x_{n}$ with weight $r$. Then $\lambda_0$ destabilizes both of $p_0$ and $q(p_0)$.

 %fixes $p_0$ and destabilizes $(x_0^d,x_1^d,\cdots,x_{r-1}^d)$ as a cycle.

% On the other hand, let $\{x^*_i\}$ be the dual basis of $\{x_i\}$. Since $\lambda_0$ acts on $x_i^*$ for $i<r$ with the character $(n-r+1)$, $\lambda_0$ also destabilizes $p_0$.  

Now, because $\rho(G)=1$ we know that there exist some integers $a,b\in\mathbb{Z}$ such that
$p_2^*M\sim ap_1^*L-bE$, where $E$ is a $p_1$-exceptional Cartier divisor. Indeed, we consider ${p_1}_{*}p_2^*M$ as a divisor on $G$ and then ${p_1}_{*}p_2^*M\sim aL$ for some $a$. Thus we denote $p_2^*M=ap_1^*L-bE$ as a divisor for some $E$ and $b$. We see that since $p_2^*M-ap_1^*L=-bE$ is $p_1$-nef, $E$ is effective and $b\ge0$ by \cite[Lemma 3.39]{komo}. Moreover, by \cite[Proposition 1.4]{GIT}, we may further assume that the above linear equivalence preserves $SL(n+1)$-linearizations. %In other words, there exists a sufficiently divisible integer $m\in\mathbb{Z}_{>0}$ such that $mp_2^*M\sim m(ap_1^*L-bE)$ as $SL(n+1)$-linearized line bundles. 

Finally, observe that because $E$ is fixed by $SL(n+1)$, we can consider $\mathcal{O}_Y(-E)$ to be a $SL(n+1)$ stable subsheaf of $\mathcal{O}_Y$. Moreover, around $p_0\in U$ we have that $G$ and $Y$ are isomorphic. Thus, letting $\tilde{p_0}\in Y$ be the corresponding point to $p_0$, we can compare $\mu^{p_2^*M}(\tilde{p_0},\lambda_0)$ with $\mu^{ap_1^*L-bE}(\tilde{p_0},\lambda_0)$ to conclude $a=1$. Note that $\lambda_0$ acts on the fiber $\mathcal{O}_Y(-E)\otimes k(\tilde{p_0})=\mathcal{O}_Y\otimes k(\tilde{p_0})$ trivially since $\tilde{p_0}\not \in E$. Therefore, (\ref{linearequiv}) holds.
  \end{proof}

 %Equality holds if and only if $y_\infty=\lim_{t\to\infty}\lambda(t)\cdot y\not\in E$.

%In this case, $\mathcal{O}_Y(-E)$ is considered to be a $SL(n+1)$-stable subsheaf of $\mathcal{O}_Y$. We may replace $Y$ by $\mathbb{A}^1$ and $E$ by $\alpha^*E$, where $\alpha:\mathbb{A}^1\to \overline{\lambda(\mathbb{G}_m)\cdot y}=y_\infty\cup\lambda(\mathbb{G}_m)\cdot y$ is the natural morphism. Then the claim immediately follows from  $$\mathrm{deg}_0E=-\mu^{\mathcal{O}_{\mathbb{A}^1}(-E)}(y,\lambda),$$ where $\mathbb{G}_m$ acts on $\mathbb{A}^1$ by multiplication and $\mathrm{deg}_0$ means the degree at $0\in\mathbb{A}^1$. Thus we complete the proof of Claim.

\bibliography{references}

\end{document}